\documentclass[a4paper,10pt]{amsart}
\usepackage[T1]{fontenc}
\usepackage[english]{babel}  
\usepackage[useregional=numeric,showdow,showzone=false]{datetime2}
\usepackage{amsmath,amssymb,amsthm}
\usepackage{xfrac}
\usepackage{amsfonts}
\usepackage{vmargin}
\usepackage{euscript} 
\usepackage{upgreek}

\DeclareMathOperator{\vol}{vol}
\DeclareMathOperator{\GL}{GL}

\theoremstyle{plain}
\newtheorem{theo}{Theorem}[section]
\newtheorem{lemma}[theo]{Lemma}

\title[Adelic approximation on spheres]{Adelic approximation on spheres} 
\author{{\'E}ric Gaudron}
\thanks{The author is supported by the ANR-23-CE40-0006-01 Gaec project. With the aim of its open access publication, he applies a CC BY open access license to this manuscript.}

\date{\DTMnow}

\begin{document}

\footnotetext{MSC~$2020$: 11R56, 11E12, 11J13.}

\selectlanguage{english}

\begin{abstract} 
We establish an adelic version of Dirichlet's approximation theorem on spheres. Let $K$ be a number field, $E$ be a rigid adelic space over $K$ and $q\colon E\to K$ be a quadratic form. Let $v$ be a place of $K$ and $\alpha\in E\otimes_{K}K_{v}$ such that $q(\alpha)=1$. We produce an explicit constant $c$ having the following property. If there exists $x\in E$ such that $q(x)=1$ then, for any $T>c$, there exists $(\upupsilon,\upphi)\in E\times K$, with $\max{(\Vert\upupsilon\Vert_{E,v},\vert\upphi\vert_{v})}\le T$ and $\max{(\Vert\upupsilon\Vert_{E,w},\vert\upphi\vert_{w})}$ controlled for any place $w$, satisfying $q(\upupsilon)=\upphi^{2}\ne 0$ and $\vert q(\alpha\upphi-\upupsilon)\vert_{v}\le c\vert\upphi\vert_{v}/T$. This remains true for some infinite algebraic extensions as well as for a compact set of places of $K$. Our statements generalize and improve on earlier results by Kleinbock \& Merrill (2015) and Moshchevitin (2017). The proofs rely on the quadratic Siegel's lemma in a rigid adelic space obtained by the author and R{\'e}mond (2017).
\end{abstract}
 
\maketitle

\section{Introduction}
Let $n\ge 1$ be an integer and $q\colon\mathbb{R}^{n}\to\mathbb{R}$ be a positive-definite quadratic form. The Euclidean variant of Dirichlet's approximation theorem asserts that, for any $\alpha\in\mathbb{R}^{n}$ and any real number $T>0$, there exists $(\upupsilon,\upphi)\in(\mathbb{Z}^{n}\times\mathbb{Z})\setminus\{0\}$ such that \[0\le\upphi\le T\quad\text{and}\quad q\left(\alpha\upphi-\upupsilon\right)\le\frac{n(\det q)^{1/n}}{T^{2/n}},\]where $\det q$  is the determinant of the symmetric matrix $A(q)$ associated to the quadratic form $q$ (in the canonical basis of $\mathbb{R}^{n}$). Its proof consists in applying Minkowski's theorem to the lattice $\mathbb{Z}^{n}\times\mathbb{Z}$ endowed with the Euclidean structure $q(\alpha\upphi-\upupsilon)+a\upphi^{2}$ for a well-chosen positive real number $a$. In 2015, Kleinbock and Merrill published a similar statement in the particular case $q(x)=x_{1}^{2}+\cdots+x_{n}^{2}$ but with the additional property $q(\upupsilon)=\upphi^{2}$ satisfied by the solution $(\upupsilon,\upphi)$~\cite{km}. In their result $q(\alpha\upphi-\upupsilon)$ is bounded by $c(n)\upphi/T$ for some positive constant $c(n)$. A generalization to any  positive-definite  quadratic form such that $A(q)\in\mathrm{M}_{n}(\mathbb{Z})$ has been achieved by Moshchevitin~\cite{mo} who, besides, gave an explicit constant depending on $n$ and $\det q$. \par The aim of this article is to improve these constants while simplifying the proofs and providing an adelic generalization. Our first result makes use of the Hermite constant $\gamma_{n}$ in dimension $n$ which is the greatest first minimum of unimodular lattices in the Euclidean space $\mathbb{R}^{n}$. 
 \begin{theo}\label{thm1}
Let $q\colon\mathbb{R}^{n}\to\mathbb{R}$ be a positive-definite quadratic form such that $A(q)$ has integral coefficients. Assume that there exists $x\in\mathbb{Q}^{n}$ such that $q(x)=1$. Then, for all $\alpha\in\mathbb{R}^{n}$ such that $q(\alpha)=1$, for all real number $T\ge (2\gamma_{n})^{n/2}\sqrt{\det q}$, there exist $\upupsilon\in\mathbb{Z}^{n}$ and $\upphi\in \mathbb{Z}$ with $1\le\upphi\le T$ satisfying\[q\left(\frac{\upupsilon}{\upphi}\right)=1\quad\text{and}\quad q\left(\alpha-\frac{\upupsilon}{\upphi}\right)\le\frac{\sqrt{8}(2\gamma_{n})^{n}\det q}{\upphi T}.\]\end{theo}
Applying this statement to \[T^{\prime}=\max{\left(T,\frac{(2\gamma_{n})^{n}\det q}{T}\right)}\] for $T>0$, which always satisfies the condition $T^{\prime}\ge (2\gamma_{n})^{n/2}\sqrt{\det q}$, leads to a variant where $T$ is only assumed to be positive. Besides the real number $\sqrt{8}(2\gamma_{n})^{n}\det q$ is smaller than the $n$th root of the constant $\kappa_{f}$ which is in~\cite[Theorem 1]{mo} but we were unable to determine whether the dependency in $\det q$ is optimal or not.\par
The proof consists in finding a small isotropic vector $(\upupsilon,\upphi)$ of the quadratic form $Q(x,y)=q(x)-y^{2}$ with the quadratic Siegel's lemma obtained by the author and R{\'e}mond~\cite[Theorem 1.2]{fq}. In order to ensure the smallness of $q(\alpha\upphi-\upupsilon)$ we twist the product norm on $\mathbb{R}^{n}\times\mathbb{R}$ with a well-chosen isometry of $Q$, which is the argument at the heart of Kleinbock and Merrill's proof (written differently). Our proof is also inspired by Moshchevitin's proof, but we avoid any choice of basis. A generalization involving an algebraic extension $K$ of $\mathbb{Q}$ and several (archimedean or ultrametric) places of $K$ will be given in \S~\ref{sec3}. The main argument of the proof is the same as the one of Theorem~\ref{thm1}, but some new difficulties appear since the condition $\upphi\ne 0$ is not automatic when $q$ is not anisotropic. That is why we prefer to start by proving this particular case.
\subsection*{Acknowledgement}I thank  Ga{\"e}l R{\'e}mond for his remarks on a previous version of this article. 

\section{Proof of Theorem~\ref{thm1}}\label{sec2}
Let $\alpha\in\mathbb{R}^{n}$ such that $q(\alpha)=1$. Denote by $b\colon\mathbb{R}^{n}\times\mathbb{R}^{n}\to\mathbb{R}$ be the symmetric bilinear form associated to $q$. 
\subsection{The Euclidean lattice}\label{subsec21} Fix a real number $t\ge 1$. For any $(x,y)\in\mathbb{R}^{n}\times\mathbb{R}$, let us consider \[\EuScript{X}=\frac{1}{2}\left(\frac{1}{t}+t\right)b(x,\alpha)+\frac{1}{2}\left(\frac{1}{t}-t\right)y\quad\text{and}\quad\EuScript{Y}=\frac{1}{2}\left(\frac{1}{t}-t\right)b(x,\alpha)+\frac{1}{2}\left(\frac{1}{t}+t\right)y\]as well as the linear map $\xi$ defined by $\xi(x,y)=\left(x-b(x,\alpha)\alpha+\EuScript{X}\alpha,\EuScript{Y}\right)$. It is an automorphism of $\mathbb{R}^{n}\times\mathbb{R}$ of determinant $1$. Indeed, since $q(\alpha)\ne 0$, the $q$-orthogonal sub\-space $\{x\in\mathbb{R}^{n}\,;\ b(x,\alpha)=0\}$ is a supplement to $\mathbb{R}.\alpha$ in $\mathbb{R}^{n}$. The choice of a basis $e_{1},\ldots,e_{n-1}$ of this hyperplane provides a basis $(e_{1},0),\ldots,(e_{n-1},0),(\alpha,0),(0,1)$ of $\mathbb{R}^{n}\times\mathbb{R}$ in which the matrix of $\xi$ is written\[\begin{pmatrix}\mathrm{I}_{n-1}& 0\\ 0 & A\end{pmatrix}\quad\text{where}\quad A=\frac{1}{2}\begin{pmatrix} \sfrac{1}{t}+t & \sfrac{1}{t}-t\\ \sfrac{1}{t}-t & \sfrac{1}{t}+t\end{pmatrix}\quad\text{has determinant $1$}.\]Thus, the Euclidean norm $\Vert(x,y)\Vert=(q(x)+y^{2})^{1/2}$ on $\mathbb{R}^{n}\times\mathbb{R}$ induces another norm $\Vert(x,y)\Vert_{t}=\Vert\xi(x,y)\Vert$. In that way, we get an Euclidean lattice $E_{t}=(\mathbb{Z}^{n}\times\mathbb{Z},\Vert\cdot\Vert_{t})$ whose volume does not depend on $t$:
\begin{lemma}\label{lem1} The volume of $E_{t}$ is equal to $\sqrt{\det q}$.
\end{lemma}
\begin{proof}The volume of $E_{t}$ is also that of $\xi(\mathbb{Z}^{n}\times\mathbb{Z})$ with respect to the norm $\Vert\cdot\Vert$ that is, $\vert\det\xi\vert\times\vol(\mathbb{Z}^{n}\times\mathbb{Z},\Vert\cdot\Vert)=\sqrt{\det q}$. \end{proof} \subsection{The quadratic form}\label{quaf}Consider the regular quadratic form $Q(x,y)=q(x)-y^{2}$ on $\mathbb{Q}^{n}\times\mathbb{Q}$ which is isotropic by hypothesis. Using that $x-b(x,\alpha)\alpha$ is $q$-orthogonal to $\alpha$, the equality $Q(\xi(x,y))=Q(x,y)$ can be checked for all $(x,y)\in\mathbb{R}^{n}\times\mathbb{R}$ with a  direct calculation. In other words:
 \begin{lemma}\label{lem2} The map $\xi$ is isometric with respect to $Q$.\end{lemma}At last, at every place $p$ of $\mathbb{Q}$, we can consider the norm $\Vert B\Vert_{p}$ of the bilinear form $B$ associated to $Q$ defined by \[\Vert B\Vert_{\infty}=\sup{\left\{\frac{\vert B((x,y),(x^{\prime},y^{\prime}))\vert}{\Vert(x,y)\Vert_{t}\Vert(x^{\prime},y^{\prime})\Vert_{t}}\mid\ (x,y),\ (x^{\prime},y^{\prime})\in(\mathbb{R}^{n}\times\mathbb{R})\setminus\{0\}\right\}}\]in the archimedean case and by $\Vert B\Vert_{p}=\max_{0\le i,j\le n}{\vert B(e_{i},e_{j})\vert_{p}}$ where $\{e_{0},\ldots,e_{n}\}$ is the canonical basis of $\mathbb{Z}^{n+1}$ and the absolute value $\vert\cdot\vert_{p}$ on $\mathbb{Q}_{p}$ normalised with $\vert p\vert_{p}=p^{-1}$. By definition, the height $H(Q)$ of $Q$ is the product of all the norms $\Vert B\Vert_{p}$ over the places $p$ of $\mathbb{Q}$. Here we have the formula $B((x,y),(x^{\prime},y^{\prime}))=b(x,x^{\prime})-yy^{\prime}$ which immediatly implies $\Vert B\Vert_{p}=1$ for all prime number $p$ since $b$ has integral coefficients. Moreover, as $\xi$ is a global isometry with respect to $Q$, we also have $\Vert B\Vert_{\infty}=\sup{\{\vert B((x,y),(x^{\prime},y^{\prime}))\vert\mid\ \Vert(x,y)\Vert=\Vert(x^{\prime},y^{\prime})\Vert=1\}}$. Then, Cauchy-Schwarz inequality applied to the positive-definite quadratic form $q(x)+y^{2}=\Vert(x,y)\Vert^{2}$ gives $\Vert B\Vert_{\infty}=1$. Hence, the height $H(Q)$ of $Q$ equals to $1$. 
  \subsection{The quadratic Siegel's lemma}\label{tqsl} We now apply~\cite[Theorem 1.3]{fq} to the quadratic space $(E_{t},Q)$ over $\mathbb{Q}$ (of dimension $n+1$) and we get $(\upupsilon,\upphi)\in\mathbb{Z}^{n}\times\mathbb{N}_{\ge 1}$ such that $Q(\upupsilon,\upphi)=0$ and \[\Vert(\upupsilon,\upphi)\Vert_{t}\le\left(2\gamma_{n}H(Q)\right)^{n/2}H(E_{t})\](the height of a vector of $E_{t}$ equals to the norm of a multiple of this element, see~\cite[Example~2 p.46]{ras}). The height $H(E_{t})$ of $E_{t}$ is nothing but the volume of $(\mathbb{Z}^{n}\times\mathbb{Z},\Vert\cdot\Vert_{t})$~\cite[p. 43]{ras} that is, $\sqrt{\det q}$ by Lemma~\ref{lem1}. Also note that the constant $c_{\mathbb{Q}}^{\mathrm{BV}}(n)$ in the original statement is simply $\gamma_{n}^{n/2}$ (see \S~\ref{subsec3.1}). As $H(Q)=1$ we get $\Vert(\upupsilon,\upphi)\Vert_{t}\le\left(2\gamma_{n}\right)^{n/2}\sqrt{\det q}$ with $q(\upupsilon/\upphi)=1$.
 \subsection{Conclusion}\label{sec:conclusion}We observe that \[q\left(\alpha-\frac{\upupsilon}{\upphi}\right)=\frac{2}{\upphi}\left(\upphi-b(\upupsilon,\alpha)\right)=\frac{2}{\upphi t}\left(\EuScript{Y}-\EuScript{X}\right)\]where $\EuScript{X},\EuScript{Y}$ are relative to $(\upupsilon,\upphi)$. Hence, since $\EuScript{X}^{2}+\EuScript{Y}^{2}\le q(\upupsilon-b(\upupsilon,\alpha)\alpha)+\EuScript{X}^{2}+\EuScript{Y}^{2}=\Vert(\upupsilon,\upphi)\Vert_{t}^{2}$, Cauchy-Schwarz inequality provides the bound $q\left(\alpha-\frac{\upupsilon}{\upphi}\right)\le2\sqrt{2}\Vert(\upupsilon,\upphi)\Vert_{t}/\upphi t$. Next we note that, since $t\ge 1$, \[\upphi=\frac{1}{2}\left(t-\frac{1}{t}\right)\EuScript{X}+\frac{1}{2}\left(t+\frac{1}{t}\right)\EuScript{Y}\le t\max{\left(\vert\EuScript{X}\vert,\vert\EuScript{Y}\vert\right)}\le t\Vert(\upupsilon,\upphi)\Vert_{t}.\]We replace $\Vert(\upupsilon,\upphi)\Vert_{t}$ by $\left(2\gamma_{n}\right)^{n/2}\sqrt{\det q}$ and we set $T=t\left(2\gamma_{n}\right)^{n/2}\sqrt{\det q}$ to end the proof of Theorem~\ref{thm1}.
 
 \section{Extended statement}
 In Theorem~\ref{thm1} the form $q$ plays two distinct roles since it is used both to define the set (ellipsoid) where the approximation takes place and to measure the quality of approximation. We can give a more general statement in which two quadratic forms appear: we will approximate points on $q=1$ using another quadratic form $q_{0}$  to measure the size of the approximation. Here it is natural to retain the hypothesis that $q_{0}$ be positive-definite but we can relax the condition on $q$, allowing some non definite forms. 
 \par Let $E$ be a vector space over a field $K$ and $q\colon E\to K$ a quadratic form. The isotropy index $i(q)$ of $q$ is the maximal dimension of totally isotropic subspaces of $q$. The induced quadratic form $Q(x,y)=q(x)-y^2$ on $E\times K$ verifies $i(Q)-i(q)\in\{0,1\}$. In fact $i(Q)=i(q)+1$ when the anisotropic part of $q$ in the Witt decomposition takes the value $1$. Given a positive-definite quadratic form $q_{0}\colon\mathbb{R}^{n}\to\mathbb{R}$, we denote by $\Vert\cdot\Vert=\sqrt{q_{0}}$ the associated Euclidean norm on $\mathbb{R}^{n}$ and by $\lambda_{1}=\min{\{\Vert\lambda\Vert\mid\lambda\in\mathbb{Z}^{n}\setminus\{0\}\}}$ the first minimum of the Euclidean lattice $(\mathbb{Z}^{n},\Vert\cdot\Vert)$. A quadratic form $q(x)={}^{\mathrm{t}}xA(q)x$ associated with a symmetric matrix $A(q)\in\mathrm{M}_{n}(\mathbb{R})$ (not necessarily positive-definite) also inherits a norm by the formula \[\Vert q\Vert_{\infty}=\max{\left\{{}^{\mathrm{t}}xA(q)y\mid x,y\in\mathbb{R}^{n},\ \Vert x\Vert=\Vert y\Vert=1\right\}}.\]In this context we have the following statement.
 \begin{theo}\label{thm2}Let $\alpha\in\mathbb{R}^{n}$ such that $q(\alpha)=1$. Define \[\EuScript{T}_{0}=n^{n/2}\left(2\max{(1,\Vert q\Vert_{\infty})}\right)^{(n-i(q))/2}\Vert\alpha\Vert\sqrt{\det q_{0}}\]and $\EuScript{T}=\max{\Big(\EuScript{T}_{0}^{1/(i(q)+1)},(\sqrt{2}/\lambda_{1})^{i(q)}\EuScript{T}_{0}\Big)}$. Assume that $A(q)\in\mathrm{M}_{n}(\mathbb{Z})$ and that $i(Q)>i(q)$ where $Q(x,y)=q(x)-y^{2}$. Then, for all real number $T\ge\EuScript{T}$, there exists $(\upupsilon,\upphi)\in\mathbb{Z}^{n}\times\mathbb{Z}$ satisfying $q(\upupsilon)=\upphi^{2}\ne 0$,\[\quad\Vert\upupsilon\Vert^{2}+(\upphi\Vert\alpha\Vert)^{2}\le\left(\Vert\alpha\Vert\Vert b(\cdot,\alpha)\Vert_{\mathrm{op},\infty}T\right)^{2}\quad\text{and}\quad\vert q(\alpha\upphi-\upupsilon)\vert\le\frac{2\sqrt{2}\EuScript{T}^{2}\upphi}{T}\times\Vert b(\cdot,\alpha)\Vert_{\mathrm{op},\infty}.\] 
 \end{theo}
 Here $\Vert b(\cdot,\alpha)\Vert_{\mathrm{op},\infty}$ denotes the operator norm of the linear form $x\mapsto b(x,\alpha)={}^{\mathrm{t}}xA(q)\alpha$ on $(\mathbb{R}^{n},\Vert\cdot\Vert)$. It can be bounded by $\Vert q\Vert_{\infty}\Vert\alpha\Vert$. If the proof of Theorem~\ref{thm2} follows the same lines as those of Theorem~\ref{thm1}, the new difficulty is to ensure $\upphi\ne 0$ even though $q$ is not assumed to be definite. To solve this problem, we introduce a maximal totally $Q$-isotropic sublattice $\Omega$ of $\mathbb{Z}^{n}\times\mathbb{Z}$ of small volume and we distinguish two cases according to the value of the first minimum  of $\Omega$. The proof is a special case of that of Theorem~\ref{thm3} (see \S~\ref{sec43}). Let us just say that the quantity $n^{n}$ in $\EuScript{T}_{0}$ is a bound for $\gamma_{i(q)+1}^{i(q)+1}\gamma_{n-i(q)}^{n-i(q)}$ which appears in $\EuScript{T}_{0}$ during the proof (see Theorem~\ref{thm3} and the discussion at the end of the article).
 
\section{Adelic generalization}
\label{sec3} 
\subsection{}\label{subsec3.1} Let $K/\mathbb{Q}$ be an algebraic extension. Its set of places has a structure of topological measure space $(V(K),\sigma)$ described in~\cite[\S~2]{cds}. For $v\in V(K)$ we denote by $K_{v}$ the topological completion of $K$ at $v$ and $\vert\cdot\vert_{v}$ is the unique absolute value on $K_{v}$ such that $\vert p\vert_{v}\in\{1,p,p^{-1}\}$ for every prime number $p$. The module $\vert f\vert$ of an integrable bounded function $f\colon V(K)\to(0+\infty)$ such that $\{v\in V(K)\,;\ f(v)\ne 1\}$ is contained in a compact subset is the positive real number \[\vert f\vert=\exp{\left(\int_{V(K)}{\log\left(f(v)\right)\,\mathrm{d}\sigma(v)}\right)}.\]Given an integer $n\ge 1$, a place $v\in V(K)$ and  $x=(x_{1},\ldots,x_{n})\in\mathbb{R}^{n}$, write \[
m_{v}(x)=\begin{cases}\left(x_{1}^{2}+\cdots+x_{n}^{2}\right)^{1/2} & \text{if $v\mid\infty$}\\ \max{\left(x_{1},\ldots,x_{n}\right)} & \text{if $v\nmid\infty$}.\end{cases}\]Then $\vert x\vert_{v}=m_{v}(\vert x_{1}\vert_{v},\ldots,\vert x_{n}\vert_{v})$ defines a norm on $K_{v}^{n}$ for all $v\in V(K)$. Let $\mathbb{A}_{K}=K\otimes_{\mathbb{Q}}\mathbb{A}_{\mathbb{Q}}$ be the ad{\`e}les of $K$. A rigid adelic space (of dimension $n$) is a $n$-dimensional $K$-vector space $E$ endowed with norms $\Vert\cdot\Vert_{E,v}$ on $E_{v}=E\otimes_{K}K_{v}$ for all $v\in V(K)$, satisfying the following property: there exist an isomorphism $\varphi\colon E\to K^{n}$ and an adelic matrix $(A_{v})_{v\in V(K)}\in\GL_{n}(\mathbb{A}_{K})$ such that \[\forall\,x\in E_{v},\quad\Vert x\Vert_{E,v}=\left\vert A_{v}\varphi_{v}(x)\right\vert_{v}\]where $\varphi_{v}=\varphi\otimes\mathrm{id}_{K_{v}}\colon E_{v}\to K_{v}^{n}$ is the natural extension of $\varphi$ to $E_{v}$. The height $H(E)$ of $E$ is the module of $v\mapsto \vert\det A_{v}\vert_{v}$ and the height $H_{E}(x)$ of $x\in E\setminus\{0\}$ is the module of $v\mapsto \Vert x\Vert_{E,v}$. The dual space $E^\vee$ has a rigid adelic structure given by the transpose map ${}^{\mathrm{t}}\varphi^{-1}\colon E^{\vee}\to K^{n}$ and $({}^{\mathrm{t}}A_{v}^{-1})_{v\in V(K)}$. Besides the product $E\times K$ has a natural rigid adelic structure given by the norms $\Vert(x,y)\Vert_{E\times K,v}=m_{v}(\Vert x\Vert_{E,v},\vert  y\vert_{v})$ for all $(x,y)\in E_{v}\times K_{v}$. \par To a rigid adelic space $E$ over $K$ can be attached its first minimum of Roy-Thunder $\lambda_{1}^{\Lambda}(E)=\inf{\left\{H_{E}(x)\,;\ x\in E\setminus\{0\}\right\}}$ and its first minimum of Bombieri-Vaaler $\lambda_{1}^{\mathrm{BV}}(E)$ which is the infimum of $r>0$ such that there exists $x\in E\setminus\{0\}$ satisfying $\sup_{v\mid\infty}{\Vert x\Vert_{E,v}}\le r$ and $\sup_{v\nmid\infty}{\Vert x\Vert_{E,v}}\le 1$ (note that $\lambda_{1}^{\Lambda}(E)\le\lambda_{1}^{\mathrm{BV}}(E)$). The last minimum $\lambda_{n}^{*}(E)$ is defined in the same way where $x$ is replaced by the vectors of a basis of $E$. They give rise to the following constants: Given $*\in\{\Lambda,\mathrm{BV}\}$ and a positive integer $n$, let us define
\[c_{K}^{*}(n)=\sup_{E}{\frac{\lambda_{1}^{*}(E)^{n}}{H(E)}}\in(0,+\infty]\]where $E$ varies among the rigid adelic spaces over $K$ of dimension $n$. Recall that Minkowski's second theorem implies $(\lambda_{1}^{*}(E))^{n-1}\lambda_{n}^{*}(E)\le c_{K}^{*}(n)H(E)$ \cite[Theorem 4.12]{cds} and, if we set $c_{1}(K)=c_{K}^{\mathrm{BV}}(1)$, then $\lambda_{i}^{\mathrm{BV}}(E)\le c_{1}(K)\Lambda_{i}(E)$ for all $1\le i\le n$ \cite[Proposition~$4.8$]{cds} and so $c_{K}^{\mathrm{BV}}(n)\le c_{1}(K)^{n}c_{K}^{\Lambda}(n)$. For all $*$, we have $c_{\mathbb{Q}}^{*}(n)=\gamma_{n}^{n/2}$ and $c_{K}^{*}(n)\le(n\delta_{K/\mathbb{Q}})^{n/2}$ when $K$ is a number field of root discriminant $\delta_{K/\mathbb{Q}}$~\cite[Proposition 5.1]{cds}. It is also known~\cite[\S~5.2]{cds} that $c_{\overline{\mathbb{Q}}}^{*}(1)=1$ and, for $n\ge 2$, \[c_{\overline{\mathbb{Q}}}^{*}(n)=\exp{\left(\frac{n}{2}\left(\frac{1}{2}+\cdots+\frac{1}{n}\right)\right)}.\]However, the constant $c_{K}^{*}(n)$ may be infinite when $n\ge 2$. For instance, this is the case when $K$ is a Northcott field of infinite degree (see Corollary 1.2 and Proposition 4.10 of~\cite{cds}). We say that $K$ is a Siegel field if $c_{K}^{\Lambda}(n)$ is finite for all $n\ge 1$.  At last, a quadratic space $(E,q)$ is a rigid adelic space $E$ endowed with a quadratic form $q\colon E\to K$. For  $v\in V(K)$ and $b\colon E\times E\to K$ the symmetric bilinear form associated to $q$, the norm $\Vert q\Vert_{v}$ is the supremum of $\vert b(x,y)\vert_{v}/\Vert x\Vert_{E,v}\Vert y\Vert_{E,v}$ for non-zero $x,y\in E\otimes_{K}K_{v}$. The height $H(q)$ of $q$ is the module of $v\mapsto \Vert q\Vert_{v}$ if $q\ne 0$ and $0$ otherwise. We also write $H(1,q)$ for the module of $v\mapsto\max{(1,\Vert q\Vert_{v})}$. 
\subsection{Main statements} Let $K/\mathbb{Q}$ be a Siegel field and let $(E,q)$ be an adelic quadratic space over $K$ of dimension $n\ge 1$. We present two statements according to the isotropy index of the quadratic form $Q(x,y)=q(x)-y^{2}$ on $E\times K$, which, as we have seen, is equal to $i(q)$ or $i(q)+1$. Let us begin with the case $i(Q)=i(q)+1$. For $v\in V(K)$, write $\epsilon_{v}=1$ if $v\mid\infty$ and $0$ otherwise.

\begin{theo}\label{thm3}Let $V\subset V(K)$ be a compact subset. Let $(\alpha_{v},t_{v})_{v\in V(K)}\in \left(E\times K\right)\otimes_{K}\mathbb{A}_{K}$ be such that $q(\alpha_{v})=1$ and $\vert t_{v}\vert_{v}>1$ for all $v\in V$. Let $\alpha\colon V(K)\to\mathbb{R}$ the function defined by $\alpha(v)=\Vert\alpha_{v}\Vert_{E,v}$ if $v\in V$ and $\alpha(v)=1$ if $v\not\in V$. Let us assume that the quadratic form $Q(x,y)=q(x)-y^{2}$ on $E\times K$ has its isotropy index $i(Q)$ equal to $i(q)+1$. Define \[\EuScript{T}_{0}=c_{K}^{\mathrm{BV}}(i(Q))c_{K}^{\Lambda}(n+1-i(Q))\left(2H(1,q)\right)^{(n+1-i(Q))/2}\vert\alpha\vert H(E)\]and 
\[\EuScript{T}=\max{\left(\EuScript{T}_{0}^{1/i(Q)},\left(\sqrt{2}/\lambda_{1}^{\mathrm{BV}}(E)\right)^{i(Q)-1}\EuScript{T}_{0}\right)}.\]We assume that $c_{1}(K)$ is finite (in particular $\EuScript{T}$ is finite).  For $v\in V$, define \[T_{v}=(2\EuScript{T})^{\epsilon_{v}}\Vert\alpha_{v}\Vert_{E,v}\Vert b(\cdot,\alpha_{v})\Vert_{E^{\vee},v}\vert t_{v}/2\vert_{v}.\]Then, for all $\varepsilon>0$, there exists $(\upupsilon,\upphi)\in E\times K$ satisfying
$q(\upupsilon)=\upphi^{2}\ne 0$ and such that:
\begin{equation}\forall\,v\not\in V,\quad m_{v}\left(\Vert\upupsilon\Vert_{E,v},\vert\upphi\vert_{v}\right)
\le \left((1+\varepsilon)\EuScript{T}\right)^{\epsilon_{v}},\end{equation}
\begin{equation}\forall\,v\in V,\quad m_{v}\left(\Vert\upupsilon\Vert_{E,v},\vert\upphi\vert_{v}\Vert\alpha_{v}\Vert_{E,v}\right)\le\left(1+\varepsilon\right)^{\epsilon_{v}}T_{v}\end{equation}and
 \begin{equation}\forall\,v\in V,\quad \left\vert q\left(\alpha_{v}\upphi-\upupsilon\right)\right\vert_{v}\le\left((1+\varepsilon)2\sqrt{2}\EuScript{T}^{2}\right)^{\epsilon_{v}}\left(\frac{\vert\upphi\vert_{v}}{T_{v}}\right)\Vert\alpha_{v}\Vert_{E,v}\Vert b(\cdot,\alpha_{v})\Vert_{E^{\vee},v}^{2}.\end{equation}
\end{theo}
The number $\Vert b(\cdot,\alpha_{v})\Vert_{E^{\vee},v}$ is the operator norm of the linear form $x\mapsto b(x,\alpha_{v})$ on $E_{v}$. Besides the number $\vert\alpha\vert$ is the module of the map $\alpha$ (see the beginning of \S~\ref{subsec3.1}). When $K$ is a number field we can take $\varepsilon=0$. A discussion about the constant of $\EuScript{T}_{0}$ is made at the end of the article.
\par Our second statement concerns the other case $i(Q)=i(q)$.

\begin{theo}\label{thm4}Consider \[\EuScript{T}_{1}=4\min{\left(c_{1}(K),\frac{c_{K}^{\mathrm{BV}}(n+1-i(Q))}{c_{K}^{\Lambda}(n+1-i(Q))}\right)}^{2}\left(\frac{\sqrt{2}}{\lambda_{1}^{\mathrm{BV}}(E)}\right)^{i(Q)}\EuScript{T}_{0}^{2}\](where $\EuScript{T}_{0}$ has been defined in the previous theorem). If $i(Q)=i(q)\ge 1$ then Theorem~\ref{thm3} remains true provided $\EuScript{T}_{0}$ which is in the definition of $\EuScript{T}$ be replaced by $\max{\{\EuScript{T}_{0},\EuScript{T}_{1}\}}$.
\end{theo}
\subsection{Preparatory statements}In this part we prove three auxiliary results useful for the proofs of Theorems \ref{thm3} and~\ref{thm4}. The notation are those of these statements.

Since only the absolute value of $t_{v}$ occurs, we can assume $t_{v}=\vert t_{v}\vert_{v}$ if $v\in V$ is archimedean. For $v\in V$, define $\EuScript{X}_{v},\EuScript{Y}_{v}\in K_{v}$  with the formulas of \S~\ref{subsec21}, where $(\alpha,t)$ is replaced by $(\alpha_{v},t_{v})$, and $b\colon E\times E\to K$ is still the symmetric bilinear form associated to $q$. Let $E_{t}$ be the rigid adelic space $E\times K$ where each norm at $v\in V$ has been twisted in the following way:
 \[\forall(x,y)\in E_{v}\times K_{v},\quad\Vert(x,y)\Vert_{E_{t},v}=m_{v}\left(\Vert x-b(x,\alpha_{v})\alpha_{v}+\EuScript{X}_{v}\alpha_{v}\Vert_{E,v},\Vert\EuScript{Y}_{v}\alpha_{v}\Vert_{E,v}\right)\](when $v\not\in V$, we have $\Vert(x,y)\Vert_{E_{t},v}=m_{v}\left(\Vert x\Vert_{E,v},\vert y\vert_{v}\right)=\Vert(x,y)\Vert_{E\times K,v}$). So, to build this norm, we  first modify the norm on $E\times K$ at $v$ by multiplying the second component by $\alpha(v)$ and then we compose with the automorphism $\xi_{v}(x,y)=\left(x-b(x,\alpha_{v})\alpha_{v}+\EuScript{X}_{v}\alpha_{v},\EuScript{Y}_{v}\right)$ of $E_{v}\times K_{v}$, which has determinant $1$. In particular we have $H(E_{t})=\vert\alpha\vert H(E\times K)=\vert\alpha\vert H(E)$.  Here are two properties of the norm $\Vert\cdot\Vert_{E_{t},v}$.
 \begin{lemma}\label{lemma32} For all $v\in V$, $x\in E\otimes_{K}K_{v}$ and $y\in K_{v}$, we have \[m_{v}(\Vert x\Vert_{E,v},\Vert y\alpha_{v}\Vert_{E,v})\le2^{\epsilon_{v}}\vert t_{v}/2\vert_{v}\left(\Vert b(\cdot,\alpha_{v})\Vert_{E^{\vee},v}\Vert\alpha_{v}\Vert_{E,v}\right)\Vert(x,y)\Vert_{E_{t},v}.\]
 \end{lemma}
 \begin{proof}The question is to bound the operator norm of $\xi_{v}^{-1}$ when $E_{v}\times K_{v}$ is endowed with the norm $m_{v}(\Vert x\Vert_{E,v},\Vert y\alpha_{v}\Vert_{E,v})$. For any $(x,y)\in E_{v}\times K_{v}$, we have the formula 
$\xi_{v}^{-1}(x,y)=\left(x-b(x,\alpha_{v})\alpha_{v}+\EuScript{X}_{v}^{\prime}\alpha_{v},\EuScript{Y}_{v}^{\prime}\right)$ where \[\EuScript{X}_{v}^{\prime}=\frac{1}{2}\left(t_{v}+\frac{1}{t_{v}}\right)b(x,\alpha_{v})+\frac{1}{2}\left(t_{v}-\frac{1}{t_{v}}\right)y\quad\text{and}\quad\EuScript{Y}_{v}^{\prime}=\frac{1}{2}\left(t_{v}-\frac{1}{t_{v}}\right)b(x,\alpha_{v})+\frac{1}{2}\left(t_{v}+\frac{1}{t_{v}}\right)y.\]When $v\in V$ is ultrametric, we have  \[\max{\left(\Vert\EuScript{X}_{v}^{\prime}\alpha_{v}\Vert_{v},\Vert\EuScript{Y}_{v}^{\prime}\alpha_{v}\Vert_{E,v}\right)}\le\max{\left(\Vert x\Vert_{E,v},\Vert y\alpha_{v}\Vert_{E,v}\right)}\times\vert t_{v}/2\vert_{v}\Vert b(\cdot,\alpha_{v})\Vert_{E^{\vee},v}\Vert\alpha_{v}\Vert_{E,v}\]since  $\vert t_{v}\vert_{v}\ge 1$ and \begin{equation}\label{equn}\vert b(x,\alpha_{v})\vert_{v}\le\Vert b(\cdot,\alpha_{v})\Vert_{E^{\vee},v}\Vert x\Vert_{E,v}\quad\text{and}\quad 1=\vert b(\alpha_{v},\alpha_{v})\vert_{v}\le\Vert b(\cdot,\alpha_{v})\Vert_{E^{\vee},v}\Vert\alpha_{v}\Vert_{E,v}.\end{equation}We easily deduce that the same bound holds for  $\max{\left(\Vert x-b(x,\alpha_{v})\alpha_{v}+\EuScript{X}_{v}^{\prime}\alpha_{v}\Vert_{E,v},\Vert\EuScript{Y}_{v}^{\prime}\alpha_{v}\Vert_{E,v}\right)}$, which gives the desired result.  
 When $v$ is archimedean, observe that \[m_{v}\left(\Vert x-b(x,\alpha_{v})\alpha_{v}+\EuScript{X}_{v}^{\prime}\alpha_{v}\Vert_{E,v},\Vert\EuScript{Y}_{v}^{\prime}\alpha_{v}\Vert_{E,v}\right)^{2}\le\left(\Vert x\Vert_{E,v}+\Vert\left(\EuScript{X}_{v}^{\prime}-b(x,\alpha_{v})\right)\alpha_{v}\Vert_{E,v}\right)^{2}+\Vert\EuScript{Y}_{v}^{\prime}\alpha_{v}\Vert_{E,v}^{2}.\]We note \[\Vert\left(\EuScript{X}_{v}^{\prime}-b(x,\alpha_{v})\right)\alpha_{v}\Vert_{E,v}\le\frac{1}{2}\left(t_{v}+\frac{1}{t_{v}}-2\right)\Vert b(x,\alpha_{v})\alpha_{v}\Vert_{E,v}+\frac{1}{2}\left(t_{v}-\frac{1}{t_{v}}\right)\Vert y\alpha_{v}\Vert_{E,v}\]and \[\Vert\EuScript{Y}_{v}^{\prime}\alpha_{v}\Vert_{E,v}\le\frac{1}{2}\left(t_{v}-\frac{1}{t_{v}}\right)\Vert b(x,\alpha_{v})\alpha_{v}\Vert_{E,v}+\frac{1}{2}\left(t_{v}+\frac{1}{t_{v}}\right)\Vert y\alpha_{v}\Vert_{E,v}.\]Then, using~\eqref{equn}, we can factorize by the product $\Vert b(\cdot,\alpha_{v})\Vert_{E^{\vee},v}\Vert\alpha_{v}\Vert_{E,v}$ and we see that \[\left(2m_{v}\left(\Vert x-b(x,\alpha_{v})\alpha_{v}+\EuScript{X}_{v}^{\prime}\alpha_{v}\Vert_{E,v},\Vert\EuScript{Y}_{v}^{\prime}\alpha_{v}\Vert_{E,v}\right)/\Vert b(\cdot,\alpha_{v})\Vert_{E^{\vee},v}\Vert\alpha_{v}\Vert_{E,v}\right)^{2}\] is bounded by  
 \[\left(\left(t_{v}+\frac{1}{t_{v}}\right)\Vert x\Vert_{E,v}+\left(t_{v}-\frac{1}{t_{v}}\right)\Vert y\alpha_{v}\Vert_{E,v}\right)^{2}+\left(\left(t_{v}-\frac{1}{t_{v}}\right)\Vert x\Vert_{E,v}+\left(t_{v}+\frac{1}{t_{v}}\right)\Vert y\alpha_{v}\Vert_{E,v}\right)^{2}.\]We develop this expression and substitute the product $2\Vert x\Vert_{E,v}\Vert y\alpha_{v}\Vert_{E,v}$ by $\Vert x\Vert_{E,v}^{2}+\Vert y\alpha_{v}\Vert_{E,v}^{2}$ to finally obtain the desired bound $4t_{v}^{2}m_{v}\left(\Vert x\Vert_{E,v},\Vert y\alpha_{v}\Vert_{E,v}\right)^{2}$.
 \end{proof}
 When $y=0$ we can prove a better estimate, which does not depend on $t_{v}$.
 \begin{lemma}\label{lemma33}
 For all $v\in V(K)$ and $x\in E\otimes_{K}K_{v}$, we have $\Vert x\Vert_{E,v}\le2^{\epsilon_{v}/2}\Vert(x,0)\Vert_{E_{t},v}$.
 \end{lemma}
 \begin{proof}If $v\in V(K)\setminus V$ then $\Vert(x,0)\Vert_{E_{t},v}=\Vert(x,0)\Vert_{E\times K,v}=\Vert x\Vert_{E,v}$ and the result is clear.  Let $v\in V$ be an archimedean place and $x\in E\otimes_{K}K_{v}$. From the definition of the $v$-norm of $E_{t}$, the number $\Vert(x,0)\Vert_{E_{t},v}^{2}$ equals \[\Big\Vert x-b(x,\alpha_{v})\alpha_{v}+\frac{1}{2}\left(t_{v}+\frac{1}{t_{v}}\right)b(x,\alpha_{v})\alpha_{v}\Big\Vert_{E,v}^{2}+\Big\vert\frac{1}{2}\Big(t_{v}-\frac{1}{t_{v}}\Big)b(x,\alpha_{v})\Big\vert_{v}^{2}\Vert\alpha_{v}\Vert_{E,v}^{2}.\]Put $\theta=\frac{1}{2}\left(t_{v}+\frac{1}{t_{v}}-2\right)b(x,\alpha_{v})\alpha_{v}$ and bound from below the first norm by $\vert\Vert x\Vert_{E,v}-\Vert\theta\Vert_{E,v}\vert$ (reverse triangle inequality). Also note that $t_{v}-1/t_{v}\ge t_{v}+1/t_{v}-2\ge 0$ since $t_{v}$ is a real number greater than $1$. In particular the norm of $(t_{v}-1/t_{v})b(x,\alpha_{v})\alpha_{v}/2$ is greater or equal than $\Vert\theta\Vert_{E,v}$. We conclude with \[\Vert(x,0)\Vert_{E_{t},v}^{2}\ge \left(\Vert x\Vert_{E,v}-\Vert\theta\Vert_{E,v}\right)^{2}+\Vert\theta\Vert_{E,v}^{2}\ge\frac{\Vert x\Vert_{E,v}^{2}}{2}.\]When $v\in V$ is ultrametric, the norm $\Vert(x,0)\Vert_{E_{t},v}$ is  \[\max{\left(\Big\Vert x-b(x,\alpha_{v})\alpha_{v}+\frac{1}{2}\Big(t_{v}+\frac{1}{t_{v}}\Big)b(x,\alpha_{v})\alpha_{v}\Big\Vert_{E,v},\Big\Vert\frac{1}{2}\Big(t_{v}-\frac{1}{t_{v}}\Big)b(x,\alpha_{v})\alpha_{v}\Big\Vert_{E,v}\right)}.\]Since $\vert t_{v}\vert_{v}>1$ we have $\vert t_{v}-1/t_{v}\vert_{v}=\vert t_{v}\vert_{v}=\vert t_{v}+1/t_{v}-2\vert_{v}$ so that \[\Vert(x,0)\Vert_{E_{t},v}=\max{\left(\Vert x+\theta\Vert_{E,v},\Vert\theta\Vert_{E,v}\right)}\ge\Vert x\Vert_{E,v}.\]\end{proof}
 At last, we also need the following statement.
  \begin{lemma}\label{lemma34} The height of $Q$ satisfies $H(Q)\le H(1,q)$.  
 \end{lemma}
  \begin{proof} Let $B$ be the bilinear form associated to $Q$ and $v\in V(K)$. Let $x,x^{\prime}\in E_{v}$ and $y,y^{\prime}\in K_{v}$. From the expression $B((x,y),(x^{\prime},y^{\prime}))=b(x,x^{\prime})-yy^{\prime}$, we get \[\vert B((x,y),(x^{\prime},y^{\prime}))\vert_{v}\le\vert b(x,x^{\prime})\vert_{v}+\vert y\vert_{v}\vert y^{\prime}\vert_{v}\le\Vert q\Vert_{v}\Vert x\Vert_{E,v}\Vert x^{\prime}\Vert_{E,v}+\vert y\vert_{v}\vert y^{\prime}\vert_{v}\](the sum can be replaced by a maximum when $v$ is ultrametric). We factorize by $\max{(1,\Vert q\Vert_{v})}$ and we use the Cauchy inequality to obtain \[\vert B((x,y),(x^{\prime},y^{\prime}))\vert_{v}\le\max{(1,\Vert q\Vert_{v})}\Vert(x,y)\Vert_{E\times K,v}\Vert(x^{\prime},y^{\prime})\Vert_{E\times K,v},\]which implies $\Vert Q\Vert_{v}\le\max{(1,\Vert q\Vert_{v})}$ when $v\not\in V$ since, in this case, $\Vert\cdot\Vert_{E_{t},v}=\Vert\cdot\Vert_{E\times K,v}$. When $v\in V$, we observe that $B((x,y),(x^{\prime},y^{\prime}))=b(x,x^{\prime})-b(y\alpha_{v},y^{\prime}\alpha_{v})$ and so \[\vert B((x,y),(x^{\prime},y^{\prime}))\vert_{v}\le\Vert q\Vert_{v}m_{v}(\Vert x\Vert_{E,v},\Vert y\alpha_{v}\Vert_{E,v})m_{v}(\Vert x^{\prime}\Vert_{E,v},\Vert y^{\prime}\alpha_{v}\Vert_{E,v}).\]Now, since $B$ is invariant by $\xi_{v}$ (Lemma~\ref{lem2}), if we replace $(x,y)$ and $(x^{\prime},y^{\prime})$ by their images by $\xi_{v}$, we deduce $\vert B((x,y),(x^{\prime},y^{\prime}))\vert_{v}\le\Vert q\Vert_{v}\Vert(x,y)\Vert_{E_{t},v}\Vert(x^{\prime},y^{\prime})\Vert_{E_{t},v}$ then $\Vert Q\Vert_{v}\le\Vert q\Vert_{v}$. Thus, in all cases we have  $\Vert Q\Vert_{v}\le\max{(1,\Vert q\Vert_{v})}$ which leads to $H(Q)\le H(1,q)$.
 \end{proof}
 
 \subsection{Proof of Theorem~\ref{thm3}}\label{sec43} According to~\cite[Corollary~$3.2$]{fq}, there exists a maximal $Q$-isotropic subspace $\{0\}\ne F\subset E_{t}$ (of dimension $i(Q)$) such that $(1+\varepsilon)^{-1/2n}\le 2H(Q)\Lambda_{1}(E_{t}/F)^{2}$. Bounding from above $\Lambda_{1}(E_{t}/F)$ by $\left(c_{K}^{\Lambda}(n+1-i(Q))H(E_{t}/F)\right)^{1/(n+1-i(Q))}$ and using $H(E_{t}/F)=H(E_{t})/H(F)=\vert\alpha\vert H(E)/H(F)$, we deduce the upper bound 
\[H(F)\le(1+\varepsilon)^{1/2}c_{K}^{\Lambda}(n+1-i(Q))\left(2H(Q)\right)^{(n+1-i(Q))/2}\vert\alpha\vert H(E)\]which leads to $H(F)\le(1+\varepsilon)^{1/2}\EuScript{T}_{0}/c_{K}^{\mathrm{BV}}(i(Q))$ with Lemma~\ref{lemma34}. Since $i(Q)>i(q)$, we have $F\not\subset E\times\{0\}$. To build the vector $(\upupsilon,\upphi)$ of Theorem~\ref{thm3}, we distinguish two cases. \par $(i)$ If $\lambda_{1}^{\mathrm{BV}}(F)<\lambda_{1}^{\mathrm{BV}}(E)/\sqrt{2}$, we consider $0<\varepsilon^{\prime}\le\varepsilon$ such that $(1+\varepsilon^{\prime})^{1/2}\lambda_{1}^{\mathrm{BV}}(F)<\lambda_{1}^{\mathrm{BV}}(E)/\sqrt{2}$ and $(\upupsilon,\upphi)\in F\setminus\{0\}$ such that $\Vert(\upupsilon,\upphi)\Vert_{E_{t},v}\le\left((1+\varepsilon^{\prime})^{1/2}\lambda_{1}^{\mathrm{BV}}(F)\right)^{\epsilon_{v}}$ for all $v\in V(K)$. By Lemma~\ref{lemma33} and the choice of $\varepsilon^{\prime}$, we have $\upphi\ne 0$. Moreover $\lambda_{1}^{\mathrm{BV}}(F)\le\left(c_{K}^{\mathrm{BV}}(i(Q))H(F)\right)^{1/i(Q)}$, so that $\Vert(\upupsilon,\upphi)\Vert_{E_{t},v}\le\left((1+\varepsilon)\EuScript{T}\right)^{\epsilon_{v}}$ for all $v\in V(K)$.
  \par $(ii)$ If $\lambda_{1}^{\mathrm{BV}}(F)\ge\lambda_{1}^{\mathrm{BV}}(E)/\sqrt{2}$, we consider $(\upupsilon,\upphi)\in F$ such that $\upphi\ne 0$ and $\Vert(\upupsilon,\upphi)\Vert_{E_{t},v}\le\left((1+\varepsilon)^{1/2}\lambda_{i(Q)}^{\mathrm{BV}}(F)\right)^{\epsilon_{v}}$ for all $v\in V(K)$. We can do that since $F\not\subset E\times\{0\}$ and every basis of $F$ contains a vector whose last coordinate is non-zero. We bound \[\lambda_{i(Q)}^{\mathrm{BV}}(F)\le\frac{c_{K}^{\mathrm{BV}}(i(Q))H(F)}{\lambda_{1}^{\mathrm{BV}}(F)^{i(Q)-1}}\le(1+\varepsilon)^{1/2}\EuScript{T}.\]Thus, in both cases, there exists $(\upupsilon,\upphi)\in E\times K$ such that $\upphi\ne 0$, $q(\upupsilon)=\upphi^{2}$ (because $(\upupsilon,\upphi)\in F$ is $Q$-isotropic) and $\Vert(\upupsilon,\upphi)\Vert_{E_{t},v}\le\left((1+\varepsilon)\EuScript{T}\right)^{\epsilon_{v}}$ for all $v\in V(K)$.  These inequalities yield the first assertion of Theorem~\ref{thm3}, since when $v\not\in V$, we have $\Vert(\upupsilon,\upphi)\Vert_{E_{t},v}=\Vert(\upupsilon,\upphi)\Vert_{E\times K,v}=m_{v}\big(\Vert\upupsilon\Vert_{E,v},\vert\upphi\vert_{v}\big)$. Now, let us consider $v\in V$. The second assertion of Theorem~\ref{thm3} is a direct consequence of Lemma~\ref{lemma32} and the definition of  $T_{v}$. At last, for $(3)$, note that $q(\alpha_{v}\upphi-\upupsilon)=2(\EuScript{Y}_{v}-\EuScript{X}_{v})\upphi/t_{v}$. From \[\vert\EuScript{X}_{v}\vert_{v}=\vert b\left(\upupsilon-b(\upupsilon,\alpha_{v})\alpha_{v}+\EuScript{X}_{v}\alpha_{v},\alpha_{v}\right)\vert_{v}\le\Vert b(\cdot,\alpha_{v})\Vert_{E^{\vee},v}\Vert\upupsilon-b(\upupsilon,\alpha_{v})\alpha_{v}+\EuScript{X}_{v}\alpha_{v}\Vert_{E,v}\]and $\vert\EuScript{Y}_{v}\vert_{v}=\vert b(\EuScript{Y}_{v}\alpha_{v},\alpha_{v})\vert_{v}\le\Vert b(\cdot,\alpha_{v})\Vert_{E^{\vee},v}\Vert\EuScript{Y}_{v}\alpha_{v}\Vert_{E,v}$ we deduce
\[\vert\EuScript{Y}_{v}-\EuScript{X}_{v}\vert_{v}\le(\sqrt{2})^{\epsilon_{v}}\Vert b(\cdot,\alpha_{v})\Vert_{E^{\vee},v}\Vert(\upupsilon,\upphi)\Vert_{E_{t},v}\]and so \[\left\vert q\left(\alpha_{v}\upphi-\upupsilon\right)\right\vert_{v}\le\left\vert\frac{2\upphi}{t_{v}}\right\vert_{v}\Vert b(\cdot,\alpha_{v})\Vert_{E^{\vee},v}\left((1+\varepsilon)\sqrt{2}\EuScript{T}\right)^{\epsilon_{v}}.\]We conclude with the formula linking $t_{v}$ and $T_{v}$.
\qed
\par\vspace{0.2cm}
 Theorem~\ref{thm2} can be deduced from Theorem~\ref{thm3}: Choose $K=\mathbb{Q}$ and the singleton $V=\{\infty\}$ (archimedean place of $\mathbb{Q}$). Take $E=\mathbb{Q}^{n}$ with the norms $\vert\cdot\vert_{p}$ at $p\in V(\mathbb{Q})\setminus\{\infty\}$ and $\Vert x\Vert_{E,\infty}=\sqrt{q_{0}(x)}$. In other words $E$ corresponds to the Euclidean lattice $(\mathbb{Z}^{n},\sqrt{q_{0}})$ and, in particular, we have $H(E)=\sqrt{\det q_{0}}$ and $\lambda_{1}^{\mathrm{BV}}(E)=\lambda_{1}$. The integrality hypothesis on the coefficients of $A(q)$ gives $\Vert q\Vert_{p}\le 1$ for all $p\ne\infty$ so that $H(1,q)= \max{(1,\Vert q\Vert_{\infty})}$.  Also choose $t_{\infty}=T/\EuScript{T}$. The equality $i(Q)=i(q)+1$ as well as the bound $c_{\mathbb{Q}}^{\mathrm{BV}}(i(q)+1)c_{\mathbb{Q}}^{\Lambda}(n-i(q))\le n^{n}$ (see \S~\ref{sec46}) allow to conclude. Theorem~\ref{thm1} also follows from Theorem~\ref{thm3} in the same way by further choosing $q_{0}=q$. This implies $\vert\alpha\vert=\Vert\alpha\Vert_{E,\infty}=\Vert b(\cdot,\alpha)\Vert_{E^{\vee},\infty}=\Vert q\Vert_{\infty}=H(1,q)=1$. Besides, since $i(Q)=1$, we have both $\EuScript{T}=\EuScript{T}_{0}$ and $c_{\mathbb{Q}}^{\mathrm{BV}}(i(Q))c_{\mathbb{Q}}^{\Lambda}(n+1-i(Q))=c_{\mathbb{Q}}^{\Lambda}(n)=\gamma_{n}^{n/2}$.

\subsection{Proof of Theorem~\ref{thm4}}The proof of Theorem~\ref{thm3} still works the same way when the maximal $Q$-isotropic subspace $F$ introduced at the beginning of the proof satisfies $F\not\subset E\times\{0\}$. Thus, in this case, Theorem~\ref{thm4} is proved. So now we can assume that $F\subset E\times\{0\}$. We shall consider the image of $F$ by a certain $Q$-isometry, image not contained in $E\times\{0\}$, with which we shall apply the same method as Theorem~\ref{thm3}. More precisely,  we claim that there exist $a\in E_{t}/F$ and a maximal $Q$-isotropic subspace $F_{a}\subset E_{t}$ satisfying the following three conditions: $(i)$ $F_{a}\not\subset E\times\{0\}$, $(ii)$ $H_{E_{t}/F}(a)\le 2(1+\varepsilon)^{1/4n}\lambda_{n+1-i(Q)}^{\mathrm{BV}}(E_{t}/F)$ and $(iii)$ $H(F_{a})\le 2H(Q)H_{E_{t}/F}(a)^{2}H(F)$. Indeed, the space $F_{a}$ comes from the key lemma of~\cite[\S~3]{fq}, whereas the element $a\in E_{t}/F$ is chosen in the same way as the beginning of the proof of \cite[Theorem 7.1]{fq} (page 234 with, here, $Z(I)=E\times\{0\}$). Besides, the minimum $\lambda_{n+1-i(Q)}^{\mathrm{BV}}(E_{t}/F)$ can be bounded in two different ways: either by $c_{1}(K)\Lambda_{n+1-i(Q)}(E_{t}/F)$ \cite[Proposition 4.8]{cds} and then by $c_{1}(K)c_{K}^{\Lambda}(n+1-i(Q))H(E_{t}/F)/\Lambda_{1}(E_{t}/F)^{n-i(Q)}$ or, directly, by $c_{K}^{\mathrm{BV}}(n+1-i(Q))H(E_{t}/F)/\lambda_{1}^{\mathrm{BV}}(E_{t}/F)^{n-i(Q)}$ and then by $c_{K}^{\mathrm{BV}}(n+1-i(Q))H(E_{t}/F)/\Lambda_{1}(E_{t}/F)^{n-i(Q)}$ since $\lambda_{1}^{\mathrm{BV}}(E_{t}/F)\ge\Lambda_{1}(E_{t}/F)$. In both cases, we bound from below $\Lambda_{1}(E_{t}/F)$ by $(1+\varepsilon)^{-1/4n}\left(2H(Q)\right)^{-1/2}$ (definition of $F$). Thus (with $i(Q)\ge 1$) $\lambda_{n+1-i(Q)}^{\mathrm{BV}}(E_{t}/F)$ is smaller than \[(1+\varepsilon)^{(n-1)/4n}\min{\left(c_{1}(K)c_{K}^{\Lambda}(n+1-i(Q)),c_{K}^{\mathrm{BV}}(n+1-i(Q))\right)}\left(2H(Q)\right)^{(n-i(Q))/2}\frac{H(E_{t})}{H(F)}.\]This information put in the previous estimate of $H(F_{a})$ implies \begin{equation*}\begin{split}H(F_{a})\le&\,(1+\varepsilon)^{1/2}\times 4\min{\left(c_{1}(K)c_{K}^{\Lambda}(n+1-i(Q)),c_{K}^{\mathrm{BV}}(n+1-i(Q))\right)}^{2}\\ & \times\left(2H(Q)\right)^{n+1-i(Q)}\left(\vert\alpha\vert H(E)\right)^{2}/H(F).\end{split}\end{equation*}We have $H(F)\ge \lambda_{1}^{\mathrm{BV}}(F)^{i(Q)}/c_{K}^{\mathrm{BV}}(i(Q))$ and,  since $F\subset E\times\{0\}$, we also have $\lambda_{1}^{\mathrm{BV}}(F)\ge\lambda_{1}^{\mathrm{BV}}(E\times\{0\})$ so $\lambda_{1}^{\mathrm{BV}}(F)\ge\lambda_{1}^{\mathrm{BV}}(E)/\sqrt{2}$ by Lemma~\ref{lemma33}. Reporting these estimates in the previous bound for $H(F_{a})$, we obtain $H(F_{a})\le(1+\varepsilon)^{1/2}\EuScript{T}_{1}/c_{K}^{\mathrm{BV}}(i(Q))$ with Lemma~\ref{lemma34}. It is then enough to resume the demonstration of Theorem~\ref{thm3} by replacing $F$ by $F_{a}$ (and so $\EuScript{T}_{0}$ by $\EuScript{T}_{1}$) to conclude. \qed

\subsection{}\label{sec46}The constant $c_{K}^{\mathrm{BV}}(i(Q))c_{K}^{\Lambda}(n+1-i(Q))$ in $\EuScript{T}_{0}$ is finite (only) when $K$ is Siegel field with $c_{1}(K)<+\infty$. This happens e.g. when $K$ is number field or $[\overline{K}:K]\le 2$ or, also, when $K=\cup_{n}{K_{n}}$ is the union of a tower of number fields $(K_{n})_{n\in\mathbb{N}}$ of bounded root discriminants~\cite[Lemma 5.8]{cds}. Besides, the following estimates, valid for all $n\ge 1$ and $i\in\{0,\ldots,n-1\}$, may be of interest:
\begin{enumerate}\item when $K$ is a number field of root discriminant $\delta_{K/\mathbb{Q}}$, we have \[c_{K}^{\mathrm{BV}}(i+1)c_{K}^{\Lambda}(n-i)\le n^{n/2}\delta_{K/\mathbb{Q}}^{(n+1)/2},\]
\item when $K=\overline{\mathbb{Q}}$, we have \[c_{\overline{\mathbb{Q}}}^{\mathrm{BV}}(i+1)c_{\overline{\mathbb{Q}}}^{\Lambda}(n-i)\le c_{\overline{\mathbb{Q}}}^{\Lambda}(n)\le n^{n/2}.\] 
\end{enumerate}
The first one derives from $c_{K}^{\mathrm{BV}}(n)\le (n\delta_{K/\mathbb{Q}})^{n/2}$ and the bound $(i+1)^{i+1}(n-i)^{n-i}\le n^{n}$ (the function $a\mapsto (a+i)\log(a+i)-a\log a$ is increasing). The second bound is a consequence of the formula given for $c_{\overline{\mathbb{Q}}}^{*}(n)$ coupled with the estimate $aH_{a}+bH_{b}\le(a+b-1)H_{a+b-1}+1$ satisfied by the harmonic number $H_{a}=1+1/2+\cdots+1/a$ for all positive integers $a,b$ and  proven by induction on $b$.

\vskip10mm\footnotesize
 Universit{\'e} Clermont Auvergne, CNRS, LMBP, F-63000 Clermont-Ferrand France
\quad\texttt{Eric.Gaudron@uca.fr}


\begin{thebibliography}{MMM}

\bibitem[Ga1]{ow}{\'E}. Gaudron. Adelic quadratic spaces (joint work with Ga{\"e}l R{\'e}mond). Oberwolfach Rep. of the workshop \emph{Lattices and Applications in Number Theory} organized by R. Coulangeon, B. Gross and G. Nebe. Report n\textsuperscript{o}~\!3, p.~$9-11$, 2016. \texttt{DOI:10.4171/OWR/2016/3}

\bibitem[Ga2]{ras}{\'E}. Gaudron. Minima and slopes of rigid adelic spaces. Chapter 2 of {\em Arakelov geometry and diophantine applications}, edited by E. Peyre and G. R{\'e}mond. Lecture Notes in Math. 2276. Springer, Cham. 2021.  p. 37--76.

\bibitem[GR1]{cds}{\'E}.~Gaudron and G.~R\'emond. Corps de Siegel. 
\emph{J. reine angew. Math.} 726. 2017. p. 187--247.

\bibitem[GR2]{fq}{\'E}.~Gaudron and G.~R{\'e}mond. 
Espaces ad\'{e}liques quadratiques. 
\emph{Math. Proc. Cambridge Philos. Soc.} 162. 2017. p.  211--247.

\bibitem[KM]{km}D. Kleinbock and K. Merrill. Rational approximation 
on spheres. \emph{Israel J. Math.} 209. 2015. p. 293--322.

\bibitem[Mo]{mo} N. Moshchevitin. Eine {B}emerkung \"{u}ber 
positiv definite quadratische {F}ormen und rationale {P}unkte. 
\emph{Math. Z.} 285. 2017. p. 1381--1388.

\end{thebibliography}
\end{document}